\documentclass[12pt,a4paper,twoside]{amsart}
\usepackage{amsmath}
\usepackage{amssymb}
\usepackage{amsfonts}
\usepackage{a4}
\usepackage{hyperref}

\newtheorem{theorem}{Theorem}[section]
\newtheorem{lemma}[theorem]{Lemma}
\newtheorem{proposition}[theorem]{Proposition}
\newtheorem{corollary}[theorem]{Corollary}
\theoremstyle{definition}

\newtheorem{example}[theorem]{Example}

\newtheorem{remark}[theorem]{Remark}

\newcommand{\norm}[1]{\left\lVert#1\right\rVert}

\usepackage{xcolor}
\usepackage{amscd}

\numberwithin{equation}{section}

\begin{document}

\title[Characteristic functions of $p$-adic integral operators]{Characteristic functions of $p$-adic integral operators}

\author{Pavel Etingof}

\address{Department of Mathematics, MIT, Cambridge, MA 02139, USA}

\author{David Kazhdan}

\address{Einstein Institute of Mathematics, Edmond J. Safra Campus,
  Givaat Ram The Hebrew University of Jerusalem, Jerusalem, 91904,
  Israel}

\dedicatory{To Jasper Stokman on his 50th birthday with admiration}

\begin{abstract} Let $P\in \Bbb Q_p[x,y]$, $s\in \Bbb C$ with sufficiently large real part, and consider the integral operator
$$
(A_{P,s}f)(y):=\frac{1}{1-p^{-1}}\int_{\Bbb Z_p}|P(x,y)|^sf(x) |dx|
$$
 on $L^2(\Bbb Z_p)$. We show that if $P$ is homogeneous of degree $d$ then for each character $\chi$ of $\Bbb Z_p^\times$ the characteristic function $\det(1-uA_{P,s,\chi})$ of the restriction $A_{P,s,\chi}$ of $A_{P,s}$ to the eigenspace $L^2(\Bbb Z_p)_\chi$ is the $q$-Wronskian of a set of solutions of a (possibly confluent) $q$-hypergeometric equation, where $q=p^{-1-ds}$. In particular, the nonzero eigenvalues of $A_{P,s,\chi}$ are the reciprocals of the zeros of such $q$-Wronskian. 
\end{abstract} 

\maketitle

\tableofcontents

\section{Introduction}

Let $P\in \Bbb Q_p[x,y]$ be a $p$-adic polynomial and $s\in \Bbb C$. We consider the operator 
\begin{equation}\label{aps}
(A_{P,s}f)(y)=\frac{1}{1-p^{-1}}\int_{\Bbb Z_p}|P(x,y)|^s f(x)|dx|
\end{equation} 
on the Hilbert space 
$H:=L^2(\Bbb Z_p)$ when ${\rm Re}(s)$ is not too negative. In this case
$A_{P,s}$ is a well-defined Hilbert-Schmidt (in particular, compact) operator. 
Therefore, its spectrum consists of eigenvalues $\lambda_i\ne 0$ 
(if present) and also $0$, which may or may not be an eigenvalue. Moreover, if $|P(x,y)|=|P(y,x)|$
and $s\in \Bbb R$ then $A_{P,s}$ is self-adjoint, hence $\lambda_i\in \Bbb R$ and we have a spectral decomposition $H={\rm Ker}A_{P,s}\oplus \bigoplus_i H(\lambda_i)$, where $H(\lambda_i)$ is the $\lambda_i$-eigenspace of $A_{P,s}$ (\cite{L}). 
The goal of this paper is to compute $\lambda_i$. 

We mostly focus on the case when $P$ is homogeneous of degree $d$. Then $A_{P,s}$ commutes with the group of units $\Bbb Z_p^\times$, hence preserves each eigenspace $H_\chi\subset H$ of $\Bbb Z_p^\times$ corresponding to a character $\chi: \Bbb Z_p^\times\to \Bbb C^\times$. Denote the restriction of $A_{P,s}$ to $H_{\chi}$ by $A_{P,s,\chi}$. If 
${\rm Re}(s)$ is sufficiently large then $A_{P,s,\chi}$ is trace class with 
exponential decay of eigenvalues, so we can define 
its characteristic function -- the Fredholm determinant 
$$
h_{P,s,\chi}(u):=\det (1-uA_{P,s,\chi})=\prod_i (1-u\lambda_i)^{{\rm mult}(\lambda_i)},
$$
where ${\rm mult}(\lambda_i)$ is the algebraic multiplicity of $\lambda_i$. Then $h_{P,s,\chi}(u)$ is 
an entire function of Hadamard order $0$ whose zeros are the reciprocals of $\lambda_i$ (\cite{L}). 

Our main result is the following theorem.  
Let $P(x,y):=x^{d-l}y^lQ(\frac{x}{y})$, where $Q(0)=1$. 
Let $k:=k_0+\delta_{1,\chi}$, where $k_0$ is the order of the pole at $0$
of the 2-variable zeta-function $Z(Q,\chi,s,z)$ defined by \eqref{zeq}.
Define {\bf the $q$-Wronski matrix} of 
a collection of functions $f_1,...,f_k$ of a complex variable $u$ 
to be the $k$-by-$k$ matrix with entries  
$$
W(f_1,...,f_k)(u)_{ji}:=f_i(q^{j-1}u).
$$
Finally, let $\Phi_i$ be the solutions of the $q$-hypergeometric equation defined by \eqref{Phiieq}, 
and 
$$
f_i(u):=(u;q)_\infty \Phi_i(u),
$$ where 
$$
(u;q)_\infty:=\prod_{n=0}^\infty (1-uq^n);
$$ 
then $u^{-\nu_i}f_i(u)$ are entire functions taking value $1$ at the origin. 

\begin{theorem}\label{th1} The characteristic function $h_{P,s,\chi}$ 
is a limit of functions of the form
$$
h(u)=\frac{(\beta u)^{-\sum_{i=1}^k \nu_i}}{\prod_{1\le i<j\le k}(q^{\nu_j}-q^{\nu_i})}\det W(f_1,...,f_k)(\beta u).
$$ 
as $\nu_1,...,\nu_{k_0}\to +\infty$ and $\beta\to \infty$ if $k_0>0$, where $q:=p^{-1-ds}$. 
\end{theorem}  

We note that such a limit can itself be written as a Wronskian of a set $f_1,...,f_k$ of solutions of a confluent $q$-hypergeometric difference equation. Also note that the order of this difference equation is typically bigger than $k$, so that $f_1,...,f_k$ don't span the space of its meromorphic solutions over the field of $q$-elliptic functions, so that the $q$-Wronskian does not factor explicitly. As a result, the eigenvalues of $A_{P,s}$ are typically given by {\it transcendental} functions of $p^s$, in contrast with Igusa's theorem \cite{I} that integrals $\int_{\Bbb Z_p^n}|P(\bold x)|^s|d\bold x|$ for a $p$-adic polynomial $P$ are always {\it rational} functions of $p^s$. We show, however, that for special values of parameters these eigenvalues may be algebraic. 

Theorem \ref{th1} is proved in Section 4. 
Our proof of Theorem \ref{th1} provides a method of computing the function $h_{P,s,\chi}$. This method is based on realizing the operator $A_{P,s,\chi}$ as a first order $q$-difference operator on a space of analytic functions. 

By the same method, analogous results can be obtained over a general non-archimedian local field $\bold F$, and for more general integral operators in which the character $w\mapsto |w|^s$ is replaced by any multiplicative character of $\bold F$. Further, they can be extended to integral operators whose kernels are 
of the form $\prod_{i=1}^m\chi_i(P_i(x,y))$ where $P_i$ are homogeneous polynomials and $\chi_i$ are multiplicative characters, and to linear combinations of such kernels.  Since these extensions are straightforward, we will not discuss them in detail. 

\begin{remark} One of our motivations for writing this paper was the desire to understand the eigenvalues of the {\bf $p$-adic Hecke operators} defined in \cite{BK}, in particular the operators $T_x$ discussed in \cite{K}, Theorem 2, which are the Hecke operators of \cite{BK} for $\Bbb P^1$ with 4 parabolic points and $G=PGL_2$. More precisely, in this paper we study what happens for $p$-adic integral operators``generically", which helps recognize, by contrast, the special features of the situation of \cite{K},\cite{BK}. Namely, as explained in \cite{K}, Subsection 2.4, this situation 
should belong to the realm of "$p$-adic integrable systems" (a notion which hasn't yet been precisely defined), and as a result the eigenvalues of $T_x$ are algebraic numbers,\footnote{This was checked 
in 2007 by M. Vlasenko in examples (\cite{V}) but the general proof has not been written down and will be given in our forthcoming joint paper with E. Frenkel.} in contrast with the present paper, where the eigenvalues are (almost certainly) transcendental.  
\end{remark} 

The organization of the paper is as follows. In Section 2 we discuss zeta-functions of univariate $p$-adic polynomials, and define a 2-variable zeta-function of such a polynomial, which is a rational function of variables $s$ and $z$ depending on a character of $\Bbb Z_p^\times$. In Section 3 we realize the operator $A_{P,s,\chi}$ 
for homogeneous $P$ as a first order $q$-difference operator on the space of analytic functions on a disk written in terms of the 2-variable zeta-function of the univariate polynomial corresponding to $P$. This allows us to write the equation for an eigenvector 
of $A_{P,s}$ as a first order $q$-difference equation. In Section 4 we recall 
the definition and properties of $q$-hypergeometric functions and then use them to solve the $q$-difference equation for an eigenvector of $A_{P,s,\chi}$. This allows us to express the characteristic function $h_{P,s,\chi}$ as a $q$-Wronskian of a set of entire solutions of a $q$-hypergeometric equation, thereby proving Theorem \ref{th1}. In Section 5 we give examples of explicit computation of $h_{P,s,\chi}$ for specific classes of homogeneous polynomials $P$. Finally, in Section 6 we consider the non-homogeneous case. 
We first show that the operator $A_{P,s}$ is trace class for sufficiently large ${\rm Re}(s)$, and then consider an example of a non-homogeneous $P$ and show that in this example the problem of finding eigenvalues of $A_{P,s}$ leads to more complicated functional equations, namely Mahler's recursions (\cite{AS,BCCD}). 

\vskip .1in

{\bf Acknowledgments.} We thank Maxim Kontsevich and Erik Koelink for useful comments. Pavel Etingof's work was partially supported by the NSF grant DMS - 1916120. The project was supported by David Kazhdan's ERC grant N$^\circ$ 669655.

\section{Zeta-functions}

Let $\chi$ be a character of $\Bbb Z_p^\times$. Let $Q\in \Bbb Q_p[x]$ 
be a polynomial with constant term $1$ of degree $r$. 
Let $s\in \Bbb C$ with ${\rm Re}(s)>-\frac{1}{r}$ if $r\ge 1$. Define the zeta-function
$$
\zeta(Q,\chi,s):=\frac{1}{1-p^{-1}}\int_{\Bbb Z_p^\times}|Q(x)|^s\chi(x)|dx|,
$$
where $|dx|$ is the additive Lebesgue measure in which the volume of $\Bbb Z_p$ is $1$. For example, if $|Q(x)|=a$ when $|x|=1$ then $\zeta(Q,\chi,s)=a^s\delta_{1,\chi}$. The function $\zeta(Q,\chi,s)$ is a special case of the Igusa zeta-function of $Q$ and therefore is rational in $p^s$ by the result of \cite{I}.  

Write $Q_m$ for the polynomial $Q_m(x):=Q(xp^{-m})$. Then 
$$
\frac{1}{1-p^{-1}}\int_{p^{-m}\Bbb Z_p^\times}|Q(x)|^s\chi(xp^{m})|dx|=p^{m}\zeta(Q_m,\chi,s).
$$

Let $c$ be the leading coefficient of $Q$. 
If $m\ll 0$ then $|Q_m(x)|=1$ when $|x|=1$, so $\zeta(Q_m,\chi,s)=\delta_{1,\chi}$. Likewise, if $m\gg 0$, we have $|Q_m(x)|=|c|p^{mr}$ when $|x|=1$, so $\zeta(Q_m,\chi,s)=|c|^sp^{mrs}\delta_{1,\chi}$. So if $|c|=1$ and $m$ is far enough from zero then 
\begin{equation}\label{zetaform}
\zeta(Q_m,\chi,s)=p^{\max(m,0)rs}\delta_{1,\chi}.
\end{equation} 

\begin{example}\label{noroots} Assume that $Q\in \Bbb Z_p[x]$ 
and its reduction $\overline Q$ to $\Bbb F_p[x]$ 
also has degree exactly $r$ and no roots in $\Bbb F_p$ (so $r\ge 2$). 
Then $|c|=1$, so $|Q(x)|=1$ when $|x|\le 1$ and $|Q(x)|=|x|^r$ when $|x|>1$. So  
$$
\zeta(Q_m,1,s)=p^{\max(m,0)rs}\delta_{1,\chi}.
$$
for all $m$. 
\end{example} 

\begin{example}\label{deg1} Let $Q(x)=1-x$, so $r=1$. If $|x|=1$ then $|1-p^{-m}x|=1$ if $m<0$ and $|1-p^{-m}x|=p^m$ if $m>0$. Thus formula \eqref{zetaform} holds for $m\ne 0$. For $m=0$, we compute: 
$$
(1-p^{-1})\zeta(1-x,\chi,s)=
\int_{\Bbb Z_p^\times}|1-x|\chi(x)|dx|=
$$
$$
\sum_{k=2}^{p-1}\int_{k+p\Bbb Z_p}
\chi(x)|dx|+\sum_{n\ge 1}p^{-ns}\sum_{k=1}^{p-1}\int_{1+kp^n+p^{n+1}\Bbb Z_p}\chi(x)|dx|.
$$
Thus for $\chi=1$ we get 
$$
\zeta(1-x,1,s)=1-\frac{p^{-1}(1-p^{-s})}{(1-p^{-1})(1-p^{-s-1})}.
$$
If $\chi\ne 1$ then let $\ell$ be the smallest integer such that $\chi$ is trivial on $1+p^\ell \Bbb Z_p$. Then we get 
$$
\zeta(1-x,\chi,s)=-\frac{p^{-1}(1-p^{-\ell(s+1)+1})}{(1-p^{-1})(1-p^{-s-1})}.
$$
So for all $\chi$ we get 
$$
\zeta(1-x,\chi,s)=\delta_{1,\chi}-\frac{p^{-1}(1-p^{-\ell(s+1)+1})}{(1-p^{-1})(1-p^{-s-1})},
$$
where for $\chi=1$ we set $\ell=0$. 
\end{example} 

Define the {\bf 2-variable zeta function} of $Q,\chi$ by 
\begin{equation}\label{zeq} 
Z(Q,\chi,s,z):=\sum_{m\in \Bbb Z} (\zeta(Q_m,\chi,s)-\delta_{1,\chi})z^m.
\end{equation} 
It is easy to see that this series is finite in the negative direction and converges for $|z|<\min(1,p^{-r{\rm Re}(s)})$ 
to a rational function of $p^s,z$ which has the form 
\begin{equation}\label{z0eq} 
Z(Q,\chi,s,z)=Z_0(Q,\chi,s,z)+\delta_{1,\chi}\left(\frac{|c|^s}{1-p^{rs}z}-\frac{1}{1-z}\right),
\end{equation} 
where $Z_0\in \Bbb C(p^s)[z,z^{-1}]$ is a Laurent polynomial. For instance, in Example \ref{noroots} we have $Z_0=0$, while in Example \ref{deg1} $Z_0=-\frac{p^{-1}(1-p^{-\ell(s+1)+1})}{(1-p^{-1})(1-p^{-s-1})}$ is independent of $z$. However, it is clear that in general $Z_0$ may contain any integer power of $z$. 

\section{Realization of $A_{P,s,\chi}$ on analytic functions} 

Let $P\in \Bbb Q_p[x,y]$ be a homogeneous polynomial of degree $d$. 
Without loss of generality we may (and will) assume that 
$$
P(x,y)=x^{d-l}y^lQ(\tfrac{x}{y}), 
$$
where $Q$ is a polynomial with constant term $1$, of some degree 
$r\le l$. Let $s\in \Bbb C$, and consider the operator $A_{P,s,\chi}$ given by \eqref{aps} acting on $H_\chi$. We are interested in nonzero eigenvalues of $A_{P,s,\chi}$. 

Recall that for two linear endomorphisms $B,C$ of a vector space, nonzero eigenvalues of $BC$ and $CB$ are the same. Thus we may assume without loss of generality that $l=d$. 

The space $H_\chi$ can be identified with the space $\ell_2$ of sequences $\lbrace f_n\rbrace$ such that $\sum_{n\ge 0} |f_n|^2<\infty$ by the assignment
$$
f_n=p^{-\frac{n}{2}}f(p^n),\ n\ge 0.
$$ 
In terms of this presentation, we have 
$$
(A_{P,s,\chi}f)(y)=\frac{|y|^{ds}}{1-p^{-1}}\sum_{n\ge 0}p^{\frac{n}{2}} f_n\int_{|x|=p^{-n}}|Q(\tfrac{x}{y})|^s\chi(xp^{-n})|dx|.
$$
Therefore, replacing $\frac{x}{y}$ with $x$ in the integral, we obtain
$$
(A_{P,s,\chi}f)_m=p^{-\frac{m}{2}}(A_{P,s,\chi}f)(p^m)=
$$
$$
\frac{p^{-\frac{3m}{2}-dms}}{1-p^{-1}}\sum_{n\ge 0}p^{\frac{n}{2}}f_n \int_{|x|=p^{m-n}}|Q(x)|^s\chi(xp^{m-n})|dx|.
$$
So we get that 
\begin{equation}\label{operseq}
(A_{P,s,\chi}f)_m=\sum_{n\ge 0}a_{mn}f_n\text{ where }
a_{mn}:=p^{-\frac{m+n}{2}-dms}\zeta(Q_{m-n},\chi,s).
\end{equation}

\begin{example}\label{noroots1} Assume that we are in the situation of Example \ref{noroots}. Then 
$$
a_{mn}=p^{-\frac{m+n}{2}-dms}p^{\max(m-n,0)rs}=
p^{-\frac{m+n}{2}-(d-r)ms-\min(m,n)rs}. 
$$
\end{example} 

Let us represent the sequence $f_n$ by the analytic function 
$$
F(z)=\sum_{n\ge 0} f_np^{\frac{n}{2}}z^n
$$ 
in the disc $|z|<p^{-\frac{1}{2}}$. This identifies $H_\chi$ with 
the space $\mathcal H$ of such analytic functions whose boundary values are 
in $L^2$, with the norm being the usual $L^2$ norm on the boundary. Multiplying \eqref{operseq} by $p^{\frac{m}{2}}z^m$ and adding up over $m\ge 0$, we get 
$$
(A_{P,s,\chi}F)(z)=\sum_{m\ge 0,n\ge 0} p^{-\frac{n}{2}-dms}\zeta(Q_{m-n},\chi,s)z^{m-n}f_nz^n=
$$
$$
\delta_{1,\chi}\frac{F(p^{-1})}{1-p^{-ds}z}+
\sum_{m\ge 0,n\ge 0} p^{-\frac{n}{2}-dms}(\zeta(Q_{m-n},\chi,s)-\delta_{1,\chi})z^{m-n}f_nz^n=
$$
$$
\delta_{1,\chi}\frac{F(p^{-1})}{1-p^{-ds}z}+
(Z(Q,\chi,s,p^{-ds}z)F(p^{-1-ds}z))_+
$$
where $G_+$ is the regular part of a Laurent series $G$. 

Thus we obtain the following proposition.\footnote{Note that Proposition \ref{analfun} continues to hold when $d$ is not necessarily an integer. Indeed, the function 
$|y|^d|Q(x/y)|$ makes sense for non-integer $d$.}

\begin{proposition}\label{analfun}
If $l=d$ then the operator $A_{P,s}$ on $H_\chi$ is equivalent to the operator 
on $\mathcal H$ given by the formula 
$$
(A_{P,s,\chi}F)(z)=
$$
$$
\delta_{1,\chi}\left(\frac{F(p^{-1})-F(p^{-1-ds}z)}{1-p^{-ds}z}+\frac{|c|^sF(p^{-1-ds}z)}{1-p^{-(d-r)s}z}\right)+
(Z_0(Q,\chi,s,p^{-ds}z)F(p^{-1-ds}z))_+.
$$
\end{proposition} 

Conjugating this by $|x|^{d-l}$, we obtain a similar formula for general $l\in [r,d]$: 
$$
(A_{P,s,\chi}F)(z)=
$$
$$
\delta_{1,\chi}\left(\frac{F(p^{-1-(d-l)s})-F(p^{-1-ds}z)}{1-p^{-(d-l)s}z}+\frac{|c|^sF(p^{-1-ds}z)}{1-p^{-(l-r)s}z}\right)+
(Z_0(Q,\chi,s,p^{-ls}z)F(p^{-1-ds}z))_+.
$$
This immediately implies

\begin{corollary} 
If ${\rm Re}(s)>-\frac{1}{\max(2(d-l),2(l-r),d)}$ then the operator $A_{P,s,\chi}$ is well defined and trace class, with exponential decay of eigenvalues. Moreover, if $\chi\ne 1$ then this is so for ${\rm Re}(s)>-\frac{1}{d}$. 
\end{corollary} 

So we will assume from now on that ${\rm Re}(s)$ varies in this range. 

\section{$q$-hypergeometric functions and proof of Theorem \ref{th1}}

\subsection{Basic hypergeometric series} 

The general {\bf basic hypergeometric series} is defined by the formula 
$$
{}_\ell\phi_{r-1}(a_1,...,a_\ell; b_1,...,b_{r-1}; q,u)=\sum_{n=0}^\infty 
\frac{(a_1;q)_n...(a_\ell;q)_n}{(b_1;q)_n...(b_{r-1},q)_n(q;q)_n}\left((-1)^nq^{\frac{n(n-1)}{2}}\right)^{r-\ell}u^n,
$$
where $\ell\le r$ and $(a,q)_n$ is the $q$-Pochhammer symbol: 
$$
(a,q)_n:=(1-a)(1-aq)...(1-aq^{n-1}),
$$
see \cite{GR}. We assume that $|q|<1$ and $b_j$ 
are not equal to a nonpositive integer power of $q$. 
Then this series converges for $|u|<1$ if $r=\ell$ and 
for all $u$ if $r>\ell$. 

There is also Bailey's variant of the basic hypergeometric series 
$$
{}_\ell\varphi_{r-1} (a_1,...,a_\ell; b_1,...,b_{r-1}; q,u)=\sum_{n=0}^\infty 
\frac{(a_1;q)_n...(a_\ell;q)_n}{(b_1;q)_n...(b_{r-1},q)_n(q;q)_n}u^n,
$$
which converges in the disc $|u|<1$ for all $\ell,r$. 
The two versions agree if $r=\ell$, and 
${}_\ell\varphi_{r-1}$ can be obtained from 
${}_m\phi_{m-1}$ with $m=\max(\ell,r)$ by specializing some parameters to $0$. 
Also, ${}_\ell\phi_{r-1}$ can be obtained from ${}_{r}\varphi_{r-1}$ by 
sending $a_{\ell+1},...,a_{r}$ to $\infty$. 

\subsection{The $q$-hypergeometric difference equation}
It is easy to see that ${}_{\ell}\varphi_{r-1}$ satisfies the difference equation 
\begin{equation}\label{hdiff}
u(1-a_1T)...(1-a_\ell T)\Phi(u)=(1-b_1q^{-1}T)...(1-b_{r-1}q^{-1}T)(1-T)\Phi(u),
\end{equation} 
where $(T\Phi)(u):=\Phi(qu)$. This equation is called the {\bf $q$-hypergeometric difference equation}.

We will mostly use the special case $r=\ell$, when $\phi=\varphi$ and 
$$
{}_{\ell}\phi_{\ell-1}(a_1,...,a_{\ell}; b_1,...,b_{\ell-1}; q,u)=\sum_{n=0}^\infty 
\frac{(a_1;q)_n...(a_{\ell};q)_n}{(b_1;q)_n...(b_{\ell-1},q)_n(q;q)_n}u^n.
$$ 
Equation \eqref{hdiff} immediately implies that ${}_\ell\phi_{\ell-1}$ is meromorphic in the whole complex plane with at most simple poles at the points $1,q^{-1},q^{-2},...$, present for generic parameter values. 

\begin{example} The {\bf Heine hypergeometric function} 
${}_2\phi_1(a,b;c;q,u)$ satisfies the equation 
$$
u(1-aT)(1-bT)\Phi(u)=(1-cq^{-1}T)(1-T)\Phi(u).
$$
\end{example} 

One may also consider a more general $q$-hypergeometric equation 
\begin{equation}\label{hdiff1}
u(1-a_1T)...(1-a_\ell T)\Phi(u)=(1-b_1T)...(1-b_{r}T)\Phi(u).
\end{equation} 
If $a_i,b_j\ne 0$, this equation is regular at $0$ iff $r\ge \ell$ and at $\infty$ iff $r\le l$. 
If $b_j=q^{-\nu_j}$ are distinct, equation \eqref{hdiff1} has $r$ solutions with 
power behavior at $0$:
\begin{equation}\label{Phii}
\Phi_{\ell,r,i}(a_1,...,a_\ell;b_1,...,b_r;q,u):=u^{\nu_i}{}_\ell \varphi_{r-1}(\tfrac{a_1}{b_i},...,\tfrac{a_\ell}{b_i}; \tfrac{qb_1}{b_i},...,\tfrac{qb_r}{b_i}; q,u),
\end{equation} 
where the $i$-th term in the $b$-list is omitted. These solutions form a basis of solutions 
if and only if $r\ge \ell$ (or, equivalently, the equation is regular at $0$). 

\subsection{Confluent limits} 
The $q$-hypergeometric difference equation has many interesting limits when various parameters go to $0$ and $\infty$ in various coordinated ways, which are called {\bf confluent limits}. Under these limits, the $q$-hypergeometric equation turns into 
{\bf confluent hypergeometric equations} which may be irregular at $0$ and $\infty$, and 
$q$-hypergeometric functions turn into {\bf confluent $q$-hypergeometric functions}. 
For example, the $q$-hypergeometric functions ${}_\ell \varphi_{r-1}$  
and ${}_\ell \phi_{r-1}$ are confluent unless $\ell=r$. 

A basic example of a confluent $q$-hypergeometric function is the entire function 
\begin{equation}\label{functionJ} 
J(a,q,u):=\sum_{n=0}^\infty \frac{(-1)^nq^{\frac{n(n-1)}{2}}u^n}{(a;q)_n(q;q)_n}={}_1\phi_1(0;a;q,u). 
\end{equation} 
which will arise in the next section. This function satisfies the {\bf $q$-Bessel difference equation} 
$$
-uT\Phi(u)=(1-aq^{-1}T)(1-T)\Phi(u)
$$
and expresses in terms of the {\bf Hahn-Exton $q$-Bessel function}
$$
J_\nu(x;q):=x^\nu\frac{(q^{\nu+1};q)_\infty}{(q,q)_\infty}J(q^{\nu+1},q,qx^2)
$$
also known as {\bf the third Jackson $q$-Bessel function}, see \cite{KSw, KS,GR}. 

\subsection{The spectrum of a first order difference operator and proof of Theorem \ref{th1}} 

Let $\mathcal H$ be the Hilbert space of analytic functions in the disk $|z|<1$ with $L^2$ boundary values, with norm
$$
\norm{F}^2=\frac{1}{2\pi}\int_{|z|=1}|F(z)|^2d\theta. 
$$ 
For a meromorphic function $G$ at a point $a\in \Bbb C$, let 
${\rm pr}(G,a)$ be the principal part of $G$ at $a$ (a polynomial in $\frac{1}{z-a}$ with zero constant term).   
Let $q\in \Bbb C, |q|<1$. Let $R(z)\in \Bbb C(z)$ be a nonzero rational function
equipped with a subset $\lbrace{z_1,...,z_k\rbrace}$ 
of its poles in the disk $|z|<|q|^{-1}$ that we call {\bf marked},  
which includes all its poles in the smaller disk $|z|\le 1$. 

Define an operator on $\mathcal H$ by the formula 
$$
(B_{R}F)(z)=R(z)F(qz)-\sum_{j=1}^k {\rm pr}(R(z)F(qz),z_j).
$$
It is easy to see that this operator is well defined and trace class. The goal of this subsection is to compute the characteristic function $h_R(u)=\det (1-uB_R)$ of $B_R$. 
This will imply Theorem \ref{th1}, as the operator $A_{P,s,\chi}$ is a special case of $B_R$. 

We start with the case when $z_1,...,z_k$ are simple poles and $z_i\ne 0$ for all $i$. 
Let $b_i=z_i^{-1}$. Then 
$$
R(z)=\beta \frac{(1-a_1z)...(1-a_\ell z)}{(1-b_1z)...(1-b_r z)}
$$ 
for some $r\ge k$, $\beta\ne 0$. Let $b_j^{-1}R_j$ 
be the residues of $R$ at the poles $z_j$, $j\in [1,k]$; so we have 
$$
R_j=-\beta\frac{(1-\frac{a_1}{b_j})...(1-\frac{a_\ell}{b_j})}{(1-\frac{b_1}{b_j})...(1-\frac{b_r}{b_j})},
$$ 
where the $j$-th factor in the denominator is omitted. Now the operator $B_R$ takes the form 
$$
(B_RF)(z)=\beta\frac{(1-a_1z)...(1-a_\ell z)}{(1-b_1z)...(1-b_r z)}F(qz)+\sum_{j=1}^k\frac{R_jF(qb_j^{-1})}{1-b_jz}.
$$
So the eigenvalue equation for $B_R$ looks like
$$
\lambda F(z)=\beta\frac{(1-a_1z)...(1-a_\ell z)}{(1-b_1z)...(1-b_r z)}F(qz)+\sum_{j=1}^k\frac{C_j}{1-b_jz},
$$
where $C_j:=R_jF(qb_j^{-1})$. 

Consider the difference equation 
$$
F(z)=u\frac{(1-a_1z)...(1-a_\ell z)}{(1-b_1z)...(1-b_r z)}F(qz)+\frac{1}{1-b_jz}.
$$
For small $u$ it has a unique power series solution, obtained by iterating the  
equation: 
$$
F_j(z,u)=\frac{1}{1-b_jz}+u\frac{(1-a_1z)...(1-a_\ell z)}{(1-b_1z)...(1-b_r z)} \frac{1}{1-b_jqz}+...=
$$
$$
\frac{1}{1-b_jz}\cdot {}_{\ell+1}\varphi_r(a_1z,...,a_\ell z,q; b_1z,...,qb_jz,...,b_rz;q,u).
$$
So for the eigenvector we have 
$$
F(z,\lambda^{-1})=\sum_{j=1}^k\lambda^{-1}C_jF_j(z,\beta\lambda^{-1} )=
$$
\begin{equation}\label{eigvec}
\sum_{j=1}^k \frac{\lambda^{-1}  C_j}{1-b_jz}\cdot {}_{\ell+1}\varphi_r(a_1z,...,a_\ell z,q; b_1z,...,qb_jz,...,b_rz;q,\beta\lambda^{-1} ).
\end{equation} 
So we get 
$$
C_i=R_iF(qb_i^{-1})=
$$
$$
\sum_{j=1}^k \frac{\lambda^{-1}R_iC_j}{1-qb_i^{-1}b_j}\cdot {}_{\ell+1}\varphi_r(\tfrac{qa_1}{b_i},...,\tfrac{qa_\ell}{b_i},q; \tfrac{qb_1}{b_i},...,\tfrac{q^2b_j}{b_i},...,\tfrac{qb_r}{b_i};q,\beta\lambda^{-1} ). 
$$
This can be written as 
$$
\sum_{j=1}^k C_jM_{ji}(\beta\lambda^{-1} )=0,
$$
where 
$$
M_{ii}(u):={}_\ell\varphi_{r-1}(\tfrac{a_1}{b_i},...,\tfrac{a_\ell}{b_i}; \tfrac{b_1}{b_i},...,\tfrac{b_r}{b_i};q,u),
$$
 and for $j\ne i$
$$
M_{ji}(u):=\frac{{}_{\ell}\varphi_{r-1}(\tfrac{a_1}{b_i},...,\tfrac{a_\ell}{b_i}; \tfrac{b_1}{b_i},...,\tfrac{qb_j}{b_i},...,\tfrac{b_r}{b_i};q,u)-{}_{\ell}\varphi_{r-1}(\tfrac{a_1}{b_i},...,\tfrac{a_\ell}{b_i}; \tfrac{b_1}{b_i},...\tfrac{qb_j}{b_i},...,\tfrac{b_r}{b_i};q,qu)}{1-\frac{b_j}{b_i}}, 
$$
where in both cases the $i$-th term in the $b$-list is omitted. 
Hence
$$
M_{ji}(u)=\frac{u^{-\nu_i}\prod_{m\ne j}(1-b_mT)}{\prod_{m\ne i}(1-\frac{b_m}{b_i})}
\Phi_i(u),
$$
where $m$ varies in $[1,k]$ and 
\begin{equation}\label{Phiieq}
\Phi_i(u):=\Phi_{\ell,r,i}(a_1,...,a_\ell;b_1,...,b_k,q^{-1}b_{k+1},...,q^{-1}b_r;q,u)
\end{equation} 
are the solutions of the $q$-hypergeometric difference equation defined by formula \eqref{Phii} (with shifted parameters). 
Thus we have a $k$-by-$k$ matrix $M(u)$,  
and the eigenvalues of $B_R$ are solutions of the equation 
$$
\det M(\beta\lambda^{-1} )=0.
$$

The matrix $M(u)$ expresses in terms of the $q$-Wronski matrix of 
the functions $\Phi_1,...,\Phi_k$. 
Indeed, the $q$-Wronski matrix $W=W(\Phi_1,...,\Phi_k)$ has entries
$$
W_{ni}(u)=T^{n-1}\Phi_i(u).
$$
But by the Newton interpolation formula we have 
$$
T^{n-1}=\sum_{j=1}^kb_j^{1-n}\frac{\prod_{m\ne j}(1-b_mT)}{\prod_{m\ne j}(1-\frac{b_m}{b_j})}.
$$ 

Thus, setting $D:={\rm diag}(\prod_{m\ne i}(1-\tfrac{b_m}{b_i}))$ and 
$N(u):={\rm diag}(u^{-\nu_i})$, we have 
$$
M(u)=V^{-1}DW(u)N(u)D^{-1},
$$
where $V$ is the Vandermonde matrix with entries 
$$
V_{nj}=b_j^{1-n}.
$$
In particular, we find that the eigenvalues of $B_R$ are solutions of the equation 
$$
\det W(\beta\lambda^{-1})=0.
$$
If $k=r=\ell$, the matrix $W(u)$ is the $q$-Wronski matrix of a {\bf basis} of solutions 
of the $q$-hypergeometric difference equation, so its determinant factorizes, 
resulting in the eigenvalues being of the form $\frac{\prod_i a_i}{\prod_i b_i}q^n$, $n\ge 0$. This is not surprising since in this case the matrix of the operator $B_R$ is triangular, so the eigenvalues are just the diagonal entries. 

Now to obtain the characteristic function of $B_R$ we need to eliminate the poles 
of the function $\det W(\beta\lambda^{-1} )$ (i.e., ``clear denominators"). 
Since $q$-hypergeometric functions have first order poles at 
$1,q^{-1},q^{-2}$ etc., this is achieved by replacing 
$\Phi_i(u)$ by 
$$
f_i(u)=(u;q)_\infty \Phi_i(u). 
$$ 
The functions $f_i(u)$ satisfy a simple modification 
of the $q$-hypergeometric equation, and we have 

\begin{proposition}\label{Wro}  
\begin{equation}\label{equali}
\det(1-u\beta^{-1}B_R)=\frac{u^{-\sum_{i=1}^k \nu_i}}{\prod_{i<j}(q^{\nu_j}-q^{\nu_i})}\det W(f_1,...,f_k)(u).
\end{equation} 
\end{proposition} 

\begin{proof} Since $\det V=\prod_{i<j}(q^{\nu_j}-q^{\nu_i})$, the two sides of \eqref{equali} 
are entire functions of Hadamard order $0$
with the same zeros, both having value $1$ at the origin. The Proposition follows. 
\end{proof}  

\begin{remark} This method also gives explicit q-hypergeometric formulas for eigenvectors of the operator $B_R$, which are given by \eqref{eigvec}, where 
$\bold C:=(C_1,...,C_k)$ is a null row-vector of the matrix $V^{-1}DW(\beta\lambda^{-1})$. Moreover, if 
the operator $B_R$ is self-adjoint, 
the orthogonality relations for the eigenvectors 
yield interesting identities with $q$-hypergeometric functions. 
\end{remark} 

Now, if the simple poles $z_j=b_j^{-1}$, $j\in [1,k]$ of $R$ are allowed collide and produce multiple poles, including one at the origin (which would give a fully general function $R(z)$) then the characteristic function of $B_R$ is given by a confluent limit of the formula of Proposition \ref{Wro}. 
But by Proposition \ref{analfun}, the operator $A_{P,s,\chi}$ is a special case of $B_R$, with $k=k_0+\delta_{1,\chi}$, where $k_0$ is the order of pole of $Z(Q,\chi,s,z)$ at the origin, i.e., is obtained from 
the generic situation by sending $\nu_1,...,\nu_{k_0}$ to $+\infty$ and $\beta\to \infty$ if $k_0>0$. Therefore Proposition \ref{Wro} and Proposition \ref{analfun} imply Theorem \ref{th1}. 

\begin{remark} The same results with obvious changes extend over any non-archimedian local field $\bold F$ instead of $\Bbb Q_p$. Namely, $\Bbb Z_p$ should be replaced by the ring of integers $\mathcal O_{\bold F}$, $p\in \Bbb Z_p$ by a uniformizer $\pi\in \mathcal O_{\bold F}$ and $p\in \Bbb C$ by the order of the residue field $\mathcal O_{\bold F}/\pi \mathcal O_{\bold F}$ of $\bold F$. Furthermore, the character $w\mapsto |w|^s$ in the integral can be replaced with any multiplicative character $w\mapsto \eta(w):=|w|^s \eta_0(w)$, where $\eta_0$ is a character of $\mathcal O_{\bold F}^\times$, and it is easy to see that the corresponding function 
$Z_0(P,\chi,s,\eta_0,z)$ is still a Laurent polynomial of $z$. 
Finally, we can further extend this theory to kernels of the form 
$\prod_{i=1}^m \eta_i(P_i(x,y))$, and in fact to any sufficiently nice 
homogeneous kernels, as well as linear combinations of such kernels. Since these extensions are straightforward, we will not discuss them in detail. 
\end{remark} 

\subsection{The case when $R$ is a Laurent polynomial} 

To illustrate the confluent limit of Proposition \ref{Wro}, let us consider the case when $R$ is a Laurent polynomial: 
$$
R(z)=z^{-k}(1-a_1z)...(1-a_\ell z),\ k\ge 1, \ell\ge k.  
$$
This is an important special case, as it occurs 
for $A_{P,s,\chi}$ with $\chi\ne 1$.
The simplest interesting example $k=1, \ell=2$ is worked out even more explicitly 
in Section \ref{Mori}. 

For simplicity assume that 
$a_i$ are distinct; the general case is similar. Then we have 
$$
(B_RF)(z)=(R(z)F(qz))_+.
$$
To find the spectrum of $B_R$, let us introduce a parameter $t$ and let 
$$
R_t(z)=(-t)^{-k}b_1...b_k\frac{(1-a_1z)...(1-a_\ell z)}{(1-t^{-1}b_1z)...(1-t^{-1}b_kz)},
$$
for some distinct $b_i\ne 0$. Then $R_t(z)\to R(z)$ as $t\to 0$, so 
$\lim_{t\to 0}B_{R_t}\to B_R$. Thus we can find the characteristic function of $B_R$ 
by finding the characteristic function of $B_{R_t}$ and taking the limit $t\to 0$.  

In the limit $t\to 0$, the difference equation for the functions $\Phi_i$ takes the form 
\begin{equation}\label{diffequ}
(1-a_1T)...(1-a_\ell T)\Psi=u^{-1}T^k\Psi. 
\end{equation} 
which can be written as 
\begin{equation}\label{diffequ1}
(-1)^\ell a_1...a_\ell (1-a_1^{-1}T^{-1})...(1-a_\ell^{-1} T^{-1})\Psi=q^\ell u^{-1}T^{k-\ell}\Psi. 
\end{equation} 
This equation is regular at $\infty$, so it has a basis of (confluent) $q$-hypergeometric solutions with power behavior at infinity, $\Psi_1$, ..., $\Psi_\ell$, such that $\Psi_i(u)\in u^{\mu_i}\Bbb C[[u^{-1}]]$, where $\mu_i:=-\log_q  a_i$. Namely, 
$$
\Psi_i(u)=u^{\mu_i}\Psi_i^0(u),
$$
and 
$$
\Psi_i^0(u)= {}_{\ell-k}\phi_{\ell-1}\left(0,...,0;\tfrac{qa_i}{a_1},...,\tfrac{qa_i}{a_\ell};q, \frac{(-1)^kq^\ell a_i^{\ell-k}}{a_1...a_\ell u}\right).
$$
is an entire function in $u^{-1}$ equal to $1$ at $\infty$.  

Let 
$$
\theta(u)=\theta(u,q):=(q;q)_\infty(-u;q)_\infty(-qu^{-1};q)_\infty
$$
be the Jacobi theta function. We have $\theta(qu)=u^{-1}\theta(u)$.
Thus conjugating \eqref{diffequ} by $\theta(u)^{\frac{1}{k}}$, we get 
for every solution $\Psi$ of this equation: 
$$
u(1-a_1u^{\frac{1}{k}}T)...(1-a_\ell u^{\frac{1}{k}}T)\theta^{\frac{1}{k}}\Psi=(u^{\frac{1}{k}}T)^k\theta^{\frac{1}{k}}\Psi,
$$
i.e., 
$$
(1-a_1u^{\frac{1}{k}}T)...(1-a_\ell u^{\frac{1}{k}}T)\Psi_*=q^{\frac{k-1}{2}}T^k\Psi_*,
$$
where $\Psi_*:=\theta^{\frac{1}{k}}\Psi$ (using appropriate branches of the function $z^{\frac{1}{k}}$). This equation has $k$ independent solutions with power behavior near zero (in terms of the variable $u^{\frac{1}{k}}$), $E_1,...,E_k$, which can be found using the power series method. 
Namely, 
$$
E_m(u)=u^{\frac{2\pi im}{k\log q}}E(e^{\frac{2\pi im}{k}}u^{\frac{1}{k}})
$$
where 
$E$ is an entire function with $E(0)=1$. 

We have
$$
E(v)=\sum_{j=1}^\ell\xi_{j}(v)\theta(v^k)^{\frac{1}{k}}\Psi_j(v^k),
$$
where $\xi_{j}(q^{\frac{1}{k}}v)=\xi_{j}(v)$. Also the functions 
$$
\widetilde\xi_{j}(v):=v^{k\mu_j}\theta(v^k)^{\frac{1}{k}}\xi_{j}(v^k)
$$
must be meromorphic in $v$ away from the origin. 
We have 
$$
\widetilde\xi_{j}(q^{\frac{1}{k}}v)=q^{\mu_j}v^{-1}\widetilde\xi_{j}(v).
$$
Thus 
$$
\widetilde \xi_{j}(v)=\theta(q^{-\mu_j}v,q^{\frac{1}{k}})\eta_{j}(v),
$$
where $\eta_j$ is an elliptic function with period $q^{\frac{1}{k}}$. 
We conclude that 
$$
E(v)=\sum_{j=1}^\ell\eta_{j}(v)\theta(a_jv,q^{\frac{1}{k}})
\Psi_j^0(v^k).
$$

\begin{lemma}\label{cons} The elliptic functions $\eta_{j}$ are constant. 
\end{lemma}  
 
\begin{proof} Assume the contrary. 
Let $p$ be a pole of $\eta_l$ for some $l$, and assume that 
its order $r$ is maximal among all $\eta_j$ at $p$. Let $\lim_{v\to p}(v-p)^r\eta_j(v)=c_j$; so $c_l\ne 0$. Thus
$$
(v-p)^rE(vq^{\frac{n}{k}})=q^{-\frac{n(n-1)}{2k}}v^{-n}\sum_{j=1}^\ell (v-p)^r\eta_j(v)a_j^{-n}\theta(a_jv,q^{\frac{1}{k}})\Phi_j^0(q^nv^k). 
$$
So setting $v=p$, we obtain (since $E$ is entire): 
$$
 \sum_{j=1}^\ell d_ja_j^{-n}\Psi_j^0(q^np^k)=0, 
$$
$d_j:=c_j\theta(a_jp,q^{\frac{1}{k}})$. Consider these equations 
for $n=-N,-N+1,...,-N+\ell-1$ as a linear system with respect to
the unknowns $x_j:=d_ja_j^{-N}$. The matrix of this system 
has entries $\alpha_{ij}=a_j^{-i+1}\Psi_j^0(q^{-N+i-1}p^{k})$. 
When $N\to +\infty$, this matrix approaches a Vandermonde matrix (as $\Psi_j^0(\infty)=1$). Thus it is nondegenerate for $N\gg 0$, which yields that $d_j=0$, 
i.e., $c_j\theta(a_jp,q^{\frac{1}{k}})=0$ for all $j$. We conclude 
that $\theta(a_lp,q^{\frac{1}{k}})=0$, so the only pole of $\eta_l$ on $\Bbb C^\times/q^{\frac{1}{k}\Bbb Z}$
is $a_l^{-1}$. Thus $r\ge 2$ (as there are no elliptic functions with just one simple pole). 
We also see that for all $j$ with $\theta(a_jp,q^{\frac{1}{k}})\ne 0$ we have $c_j=0$, i.e., the order of the pole of $\eta_j$ at $p$ is $\le r-1$. So for such $j$ set $c_j':=\lim_{v\to p}(v-p)^{r-1}\eta_j(v)$. Then, evaluating the identity
$$
(v-p)^{r-1}E(vq^{\frac{n}{k}})=q^{-\frac{n(n-1)}{2k}}v^{-n}\sum_{j=1}^\ell (v-p)^{r-1}\eta_j(v)a_j^{-n}\theta(a_jv,q^{\frac{1}{k}})\Psi_j^0(q^nv^k) 
$$
at $p$, we obtain 
$$
\sum_{j: c_j\ne 0} c_ja_j^{-n+1}\theta'(a_jp,q^{\frac{1}{k}})\Psi_j^0(q^nv^k)+
\sum_{j: c_j= 0} c_j'a_j^{-n}\theta(a_jp,q^{\frac{1}{k}})\Psi_j^0(q^nv^k)=0,
$$
and the same argument with the Vandermonde matrix gives 
$c_l\theta'(a_lp,q^{\frac{1}{k}})=0$, hence $c_l=0$, a contradiction. 
\end{proof} 

Thus, setting $\zeta:= e^{\frac{2\pi i}{k}}$ we obtain 

\begin{proposition}\label{conss} 
$$
E(v)=\sum_{j=1}^\ell \eta_{j}\theta(a_jv,q^{\frac{1}{k}})\Psi_j^0(v^k),
$$
$\eta_{j}\in \Bbb C$, and the constants $\eta_{j}$ are uniquely determined by the condition that $E$ is holomorphic at $0$ and $E(0)=1$. 
\end{proposition} 

\begin{proof} We just have to justify uniqueness, which follows since 
the difference of two such expressions will have an inadmissible power asymptotics 
at $0$ for a solution of the corresponding difference equation. 
\end{proof}  

\begin{remark} The constants $\eta_{j}$ admit 
product formulas which can be derived by taking the confluent limit of the 
connection matrix of the $q$-hypergeometric equation. 
We will not give these formulas here except in the simplest example of the $q$-Bessel equation, see Subsection \ref{Mori}. 
\end{remark} 

Recall that $\det(\zeta^{m(n-1)})_{1\le m,n\le k}=i^{\frac{(k-1)(k-2)}{2}}k^{\frac{k}{2}}$. 
So we arrive at the following result. 

\begin{proposition} We have 
$$
\det (1-uB_R)=i^{-\frac{(k-1)(k-2)}{2}}k^{-\frac{k}{2}}\det (\zeta^{m(n-1)}T^{n-1}E(\zeta^mu^{\frac{1}{k}}))_{1\le m,n\le k},
$$
where $E$ is the entire function given by the formula of Proposition \ref{conss}. 
\end{proposition}  
 
Note that while the entries of this matrix are functions of $u^{\frac{1}{k}}$, 
the determinant is a function of $u$. 
  
\section{Examples}

In this section we consider a number of examples of computation of the characteristic function $h_{P,s,\chi}$. In all of these examples, the parameter $k$ (the size of the Wronski matrix) equals $1$, so the characteristic function turns out to be (confluent) $q$-hypergeometric. While these examples are all instances of Theorem \ref{th1}, to make the discussion more concrete, we do all computations explicitly from scratch and then compare the answers to existing literature on $q$-special functions. 

\subsection{The situation of Example \ref{noroots}} \label{noroots2}\label{ex1} Our first example is 
the setting of Example \ref{noroots}. A typical polynomial $P$ of this kind is 
$P(x,y)=y^{d-r}(x^2+y^2)^{\frac{r}{2}}$ for even $r$, where 
$p$ has remainder $3$ modulo $4$. In this case $Z_0=0$, so $A_{P,s,\chi}=0$ when 
$\chi\ne 1$, so it remains to consider the case $\chi=1$. Then, 
by Proposition \ref{analfun}, we have 
$$
(A_{P,s,1}F)(z)=
\frac{F(p^{-1})}{1-pqz}+\frac{(p^{rs}-1)pqzF(qz)}{(1-pqz)(1-p^{rs+1}qz)}. 
$$
So the eigenvalue equation has the form 
$$
\lambda F(z)=\frac{F(p^{-1})}{1-pqz}+\frac{(p^{rs}-1)pqzF(qz)}{(1-pqz)(1-p^{rs+1}qz)}. 
$$
If $r=0$ then $P=y^d$, so $A_{P,s,1}$ has rank $1$ with the only 
nonzero eigenvalue $\lambda=\frac{1}{1-q}$. The same thing happens if $s=0$. 
So let us assume that $r>0$, $s\ne 0$.  
Then we have  
$$
\lambda F(0)=F(p^{-1}).
$$
This implies that $\lambda\ne 0$. Now, if $F(p^{-1})=0$ then there are no nonzero solutions (looking at the order of $F$ at zero), so $F(p^{-1})\ne 0$. Then we may set $F(0)=1$, so that $\lambda=F(p^{-1})$. So the difference equation for the eigenvector looks like 
$$
F(z)=\frac{1}{1-pqz}+\lambda^{-1}
\frac{(p^{rs}-1)pqzF(qz)}{(1-pqz)(1-p^{rs+1}qz)}.
$$
Iterating this, we get the expression for $F$: 
$$
F(z)=\frac{1}{1-pqz}+\lambda^{-1}
\frac{(p^{rs}-1)pqz}{(1-pqz)(1-pq^2z)(1-p^{rs+1}qz)}+
$$
$$
\lambda^{-2}
\frac{(p^{rs}-1)^2p^2q^3z^2}{(1-pqz)(1-pq^2z)(1-pq^3z)(1-p^{rs+1}qz)(1-p^{rs+1}q^2z)}+...
$$
So we have 
$$
\lambda=F(p^{-1})=
\frac{1}{1-q}+\lambda^{-1}
\frac{(p^{rs}-1)q}{(1-q)(1-q^2)(1-p^{rs}q)}+
$$
$$
\lambda^{-2}
\frac{(p^{rs}-1)^2q^3}{(1-q)(1-q^2)(1-q^3)(1-p^{rs}q)(1-p^{rs}q^2)}+...
$$
Multiplying the numerators and denominators on the right hand side of this equation 
by $1-p^{rs}$ and multiplying both sides by $\lambda^{-1}$, we obtain 

\begin{proposition}\label{eig1}  
We have 
$$
h_{P,s,1}(u)=J(p^{rs},q,(1-p^{rs})u),
$$
where $J(a,q,u)={}_1\phi_1(0,a;q,u)$ is the entire function given by the formula \eqref{functionJ}. In particular, the eigenvalues of the operator $A=A_{P,s,1}$ are the solutions of the equation 
$$
J(p^{rs},q,(1-p^{rs})\lambda^{-1})=0
$$ 
and are simple. 
\end{proposition} 

Thus the characteristic function of $A_{P,s,1}$ expresses via the Hahn-Exton $q$-Bessel function with parameter
$$
\nu=-1-\frac{rs}{1+ds}. 
$$

\begin{remark} 
We obtain the following formula for the eigenvector $F(z)$ with eigenvalue $\lambda$: 
$$
F(z)=\frac{1}{1-pqz}\cdot {}_2\phi_2(0,q; pq^2z,apqz;q,(1-a)pqz\lambda^{-1}).
$$ 
where $a:=p^{rs}$. So setting $p^{\frac{1}{2}}q=b$ and replacing $z$ with $p^{-\frac{1}{2}}z$, we get that the functions 
$$
F_i(z):=\frac{1}{1-bz}\cdot {}_2\phi_2(0,q; bqz,abz;q,b^2q^{-1}zu_i),
$$
for $u_i$ roots of the equation $J(a,q,u)=0$, are orthogonal on the unit circle $|z|=1$. 
\end{remark} 

\begin{remark}\label{seq0} This calculation implies that the unbounded self-adjoint operator $(1-p^{rs})A_{P,s,1}^{-1}$ has a limit $B_0$ as $s\to 0$, which has 1-dimensional kernel. Moreover, the inverse 
$\widehat A_0$ of $B_0$ on the orthogonal complement of ${\rm Ker}B_0$ is trace class and 
has characteristic function
$$
\det (1-u\widehat A_0)=\lim_{a\to 1}(a-1)\frac{J(a,q,u)}{u}=\sum_{n=0}^\infty \frac{(-1)^nq^{\frac{n(n+1)}{2}}u^n}{(q;q)_n (q;q)_{n+1}}.
$$
\end{remark} 

Consider now the case $l=\frac{r+d}{2}$ and $s>-\frac{1}{d}$. Then the operator $A=A_{P,s,1}$ has matrix 
$$
a_{mn}=q^{\frac{m+n}{2}}\bold q^{\frac{m+n}{2}-\min(m,n)}. 
$$ 
where $\bold q:=p^{-1}q^{-1}=q^{\nu+1}$. Note that this operator is trace class whenever $|q|<1,|\bold q|<|q|^{-1}$ and self-adjoint if in addition $q,\bold q\in \Bbb R$.  

\begin{proposition} \label{posi}
If $0<\bold q,q<1$ then the operator $A$ is positive.
\end{proposition} 

\begin{proof} It suffices to show that for any $0<\bold q<1$, the matrix with entries
$$
b_{mn}:=\bold q^{-\min(m,n)}, 0\le m,n\le N
$$
is positive definite. By Sylvester's criterion, for this it suffices to show that its determinant $D_N$ is positive (as we know that $0$ is not an eigenvalue). Subtracting from the first row of this matrix $\bold q$ times the second row, we get 
$$
D_N=\bold q^{-N}(1-\bold q)D_{N-1},
$$ 
with $D_0=1$. 
So $D_N=\bold q^{\frac{-N(N+1)}{2}}(1-\bold q)^{N}>0$.  
\end{proof} 

\begin{corollary}\label{realzer} 
If $|q|<1,|\bold q|<|q|^{-1}$ then the zeros of the function $(\bold q,q)_\infty J(\bold q,q,u)$ are all simple. Moreover, for $0<q<1$, they are real if in addition $0<\bold q\le q^{-1}$, and nonnegative if $0<\bold q\le 1$. 
\end{corollary} 

\begin{proof} For $\bold q=1$ the result follows from Remark \ref{seq0}, so assume $\bold q\ne 1$. Then the first statement follows from Proposition \ref{eig1} (interpolated to non-integer values of $d$). For the second statement it suffices to assume 
that $\bold q\ne q^{-1}$. So the operator $A$ is a self-adjoint trace class operator. Hence 
its eigenvalues $\lambda_i$ are real, and by Proposition \ref{posi} they are positive if 
$\bold q<1$. But $(1-\bold q)\lambda_i^{-1}=u_i$, where $u_i$ are the zeros of $J(\bold q,q,u)$. This implies the second statement. 
\end{proof}  

\begin{remark} 1. We note that Corollary \ref{realzer} is not new and is given just for illustration purposes. Namely, the second statement of Corollary \ref{realzer} is proved in \cite{KS} using a different method. Also a proof similar to ours is given in \cite{SS}.  

2. The assumptions of Corollary \ref{realzer} cannot be 
dropped. For example, if $\bold q=q^{-m+1}$ for $m\ge 0$ (i.e., $\nu=-m$) then $(\bold q,q)_\infty J(\bold q,q,u)$ has a zero of order $m$ at $u=0$ (so a multiple zero for $m\ge 2$). This also shows that when $q,\bold q\in (0,1)$ and $\bold q$ is close to $q^{-m+1}$ then there have to be zeros of $(\bold q,q)_\infty J(\bold q,q,u)$ with argument close to $\frac{2\pi k}{m}$ or $\frac{2\pi (k+\frac{1}{2})}{m}$ (depending on the sign of $1-q^{m-1}\bold q$). So for $m\ge 2$ they cannot be all real.   

A more complete analysis of the zeros of the function $J$ is carried out in \cite{AM}. In particular, it is shown there that for $q,\bold q\in (0,1)$ the number of non-real zeros is always finite. 
\end{remark} 

\subsection{A rank $1$ perturbation}
The operator $A$ from Subsection \ref{noroots2} for $r=d$ is a rank $1$ perturbation of the inverse to a second order difference operator. Namely, consider the self-adjoint trace class operator $\widetilde A$ with kernel $\frac{|P(x,y)|^s-1}{1-p^{-ds}}$; i.e., 
$\widetilde A=\frac{1}{1-pq}(A-\frac{1}{1-p^{-1}}\Pi)$, where $\Pi$ is the projector to constant functions. This operator has a limit at $s=0$ 
where the kernel becomes $\frac{1}{d}\log|P(x,y)|$. 
As before, $\widetilde A$ vanishes on $H_\chi$ with $\chi\ne 1$, so let us consider it on $H_1\cong \ell^2$. Then the matrix of $\widetilde A$ in the standard basis is given by 
$$
\widetilde a_{mn}=p^{-\frac{m+n}{2}}\frac{p^{-\min(m,n)ds}-1}{1-p^{-ds}}.
$$ 
There is an obvious eigenvector $(1,0,...)$ with eigenvalue $0$ (the indicator function of $\Bbb Z_p^\times$), so let's consider the action of $\widetilde A$ on its orthogonal complement, 
i.e., on the space $\ell_{2,0}$ of square summable 
sequences $\lbrace f_n\rbrace$ such that $f_0=0$.
This space is identified with $\mathcal H$ by the assignment 
$F(z)=\sum_{n\ge 1}p^{\frac{n}{2}}f_nz^{n-1}$.  
In terms of this realization we have 
$$
(\widetilde A F)(z)=\frac{qzF(qz)-p^{-1}F(p^{-1})}{(1-z)(1-pqz)}.
$$
Thus the eigenvalue equation 
has the form 
$$
\lambda F(z)=\frac{qzF(qz)-p^{-1}F(p^{-1})}{(1-z)(1-pqz)}.
$$
If $\lambda=0$ then we get $qzF(qz)=p^{-1}F(p^{-1})$, i.e. $F=0$. Thus $\lambda\ne 0$.
Now, if $F(p^{-1})=0$ then there are no nonzero solutions (looking at the order of $F$ at zero), so $F(p^{-1})\ne 0$. Then we may set $F(0)=1$, so that 
$$
\lambda=-p^{-1}F(p^{-1}).
$$ 
Thus we have 
$$
F(z)=\frac{1}{(1-z)(1-pqz)}+\lambda^{-1}\frac{qz}{(1-z)(1-pqz)}F(qz).
$$
Iterating this, we get the expression for $F$: 
$$
F(z)=\frac{1}{(1-z)(1-pqz)}+\lambda^{-1}\frac{qz}{(1-z)(1-qz)(1-pqz)(1-pq^2z)}+
$$
$$
+\lambda^{-2}\frac{q^3z^2}{(1-z)(1-qz)(1-q^2z)(1-pqz)(1-pq^2z)(1-pq^3z)}+...
$$
So setting $z=p^{-1}$, we have
$$
1+\lambda^{-1}\frac{p^{-1}}{(1-p^{-1})(1-q)}+\lambda^{-2}\frac{qp^{-2}}{(1-p^{-1})(1-p^{-1}q)(1-q)(1-q^2)}
$$
$$
+\lambda^{-3}\frac{q^3p^{-3}}{(1-p^{-1})(1-p^{-1}q)(1-p^{-1}q^2)(1-q)(1-q^2)(1-q^3)}+...=0.
$$
Thus we get 

\begin{proposition} The characteristic function of $\widetilde A$ is given by the formula 
$$
\det (1-u\widetilde A)=J(p^{-1},q,-p^{-1}u),
$$
so the eigenvalues of $\widetilde A$ are the solutions of the equation
\begin{equation}\label{eigsol1} 
J(p^{-1},q,-p^{-1}\lambda^{-1})=0
\end{equation} 
and are simple. 
\end{proposition} 
Note that the Hahn-Exton parameter $\nu$ now equals $-\frac{ds}{ds+1}$. 

Consider now the inverse operator $\widetilde A^{-1}$. Let $F=\widetilde A^{-1}G$. Then we have 
$$
G(z)=\frac{qzF(qz)-p^{-1}F(p^{-1})}{(1-z)(1-pqz)}. 
$$
Thus $G(0)=-p^{-1}F(p^{-1})$. So we have 
$$
(\widetilde A^{-1}G)(z)=(1-q^{-1}z)(z^{-1}-p)G(q^{-1}z)-z^{-1}G(0)=
((1-q^{-1}z)(z^{-1}-p)G(q^{-1}z))_+.
$$
This shows that $\widetilde A^{-1}$ is an unbounded self-adjoint operator, which is a self-adjoint extension of the symmetric operator $L=L_\nu$ given by 
$$
(LG)(z):= ((1-q^{-1}z)(z^{-1}-p)G(q^{-1}z))_+
$$
with initial domain being the space of polynomials $\Bbb C[z]$ with norm 
$$
\norm{G}^2=\frac{1}{2\pi}\int_{|z|=p^{-\frac{1}{2}}}|G(z)|^2d\theta.
$$ 
Writing $G(z)$ as $\sum_{n\ge 0}p^{\frac{n}{2}}g_nz^n$, we obtain an expression for $L$ as a difference 
operator on the space of sequences whose matrix in the standard basis is a Jacobi matrix: 
$$
(Lg)_n=q^{-n-1}(p^{\frac{1}{2}}g_{n+1}-(1+pq)g_n+p^{\frac{1}{2}}qg_{n-1}),
$$
where we agree that $g_{-1}=0$. It follows that the self-adjoint extension of $L$ given by $\widetilde{A}^{-1}$ has discrete spectrum, and its eigenvalues are the solutions $u_i$ of the equation 
$$
J(p^{-1},q,-p^{-1}u)=0. 
$$

This analysis is carried out in \cite{SS}. Let us explain its results in more detail. 
Consider the homogeneous difference 
equation $Lg=0$. Its basic solutions (generically) are $g_n=\beta^n$ where $\beta$ 
is a root of the characteristic equation 
$$
p^{\frac{1}{2}}\beta^2-(1+pq)\beta+p^{\frac{1}{2}}q=0. 
$$
The roots of this equation are $\beta_1=p^{-\frac{1}{2}}$ and $\beta_2=p^{\frac{1}{2}}q=p^{-\frac{1}{2}-ds}$. So in the range $-\frac{1}{d}<s\le -\frac{1}{2d}$ (i.e., $\nu\ge 1$) we have $p^{-\frac{1}{2}-ds}\ge 1$, so the equation $Lg=0$ has a unique, up to scaling, solution in $\ell_2$ (namely, $p^{-\frac{n}{2}}$). This implies (after some work, done in \cite{SS}) that the operator $L$ is essentially self-adjoint on $\Bbb C[z]$. On the other hand, if $s>-\frac{1}{2d}$ (i.e., $|\nu|<1$) then $p^{-\frac{1}{2}-ds}<1$, so both basic solutions are in $\ell_2$. This implies (again after some work done in \cite{SS}) that the operator $L$ is not essentially self-adjoint but rather has von Neumann deficiency indices $(1,1)$. So $L$ has a 1-parameter family of self-adjoint extensions parametrized by $(\alpha_1,\alpha_2)\in \Bbb R\Bbb P^1$ obtained by adding the function $\psi_{\alpha_1,\alpha_2}:=\frac{\alpha_1}{1-z}+\frac{\alpha_2}{1-pqz}$ to the initial domain of $L$, which makes it essentially self-adjoint (if $s=0$, i.e., $q=p^{-1}$, we should add $\psi_{\alpha_1,\alpha_2}=\frac{\alpha_1}{1-z}+\frac{\alpha_2}{(1-z)^2}$). Moreover, the self-adjoint extension of $L$ given by $\widetilde A^{-1}$ corresponds to 
the point $(1,0)\in \Bbb R\Bbb P^1$. Finally, if $-\frac{1}{2d}<s\le 0$ (i.e., $0\le \nu<1$) then $L$ is positive and $\widetilde A^{-1}$ is its Friedrichs extension, as it has 
the smallest domain, see \cite{L}. This comes from the 
fact that the decay of Taylor coefficients of the function $\psi_{\alpha_1,\alpha_2}$ is fastest if and only if $\alpha_2=0$. 

For example, in the case $s=0$, i.e., $q=p^{-1}$, we get the difference operator 
$$
(Lg)_n=p^n(p^{\frac{1}{2}}g_{n+1}-2x_n+p^{-\frac{1}{2}}g_{n-1}), 
$$
i.e. $L=DD^*$ where 
$$
(Dg)_n=p^{\frac{n}{2}}(g_n-g_{n-1}). 
$$
The eigenvalues of the Friedrichs extension of $L$ are, therefore, solutions of the equation $J(p^{-1},p^{-1},-p^{-1}u)=0$. This is shown in \cite{KRS}. 

\begin{remark} A more complete account of the spectral theory of difference operators (also called $q$-Sturm-Liouville theory) can be found in the book \cite{AM2}. 
\end{remark} 

\begin{remark} 
We obtain the following formula for the eigenvector $F(z)$ with eigenvalue $\lambda$: 
$$
F(z)=\frac{1}{(1-z)(1-pqz)}\cdot {}_2\phi_2(0,q; pq^2z,pqz;q,-z\lambda^{-1}).
$$ 
Thus the eigenvectors are 
$$
F_i(z)=\frac{1}{(1-z)(1-pqz)}\cdot {}_2\phi_2(0,q; pq^2z,pqz;q,pzu_i).
$$ 
where $u_i$ are the roots of the equation $J(p^{-1},q,u)=0$. 

The fact that the inverse of $\widetilde A$ is a difference operator $L$ implies 
that we also have
$$
F_i(z)=\sum_{n=0}^\infty J(p^{-1},q,q^{n+1}u_i)z^n.
$$
Thus we obtain the orthogonality relations 
$$
\sum_{n\ge 0} p^{-n}J(p^{-1},q,q^{n+1}u_i)J(p^{-1},q,q^{n+1}u_j)=0,\ i\ne j
$$
which are the Hahn-Exton $q$-analogs of the Fourier-Bessel orthogonality relations with $p^{-1}=q^{\alpha+1}$ (see \cite{KSw}, p.2).  
\end{remark}

\subsection{Continuous limit}

We can also consider the ``archimedian" limit of the operator $\widetilde A|_{H_1}$ from the previous subsection as $p\to 1$ (this does not have a $p$-adic interpretation but makes sense analytically). For this pick $a>0$ and think of $q^{\frac{n\nu}{2a}}f_n$ as values of some function $\bold f$ at $q^{-\frac{n\nu}{a}}$. Then the limit $p\to 1$ will be a continuous limit which transforms Riemann sums into integrals over $[0,1]$. The condition $\sum_n |f_n|^2<\infty$ 
translates into the condition that 
$$
\int_0^1 |\bold f(x)|^2dx<\infty, 
$$
i.e., $\bold f\in L^2[0,1]$. Also the operator $\widetilde A$ after appropriate renormalization converges to 
\begin{equation}\label{atild}
(\widetilde A_\infty(a)\bold f)(x)=\frac{1}{a}\int_0^1 ({\rm max}(x,y)^{-a}-1)(xy)^{-\frac{a}{2ds}-\frac{1}{2}}\bold f(y)dy. 
\end{equation} 
Note that the operator $\widetilde A_\infty(a)$ is unitarily equivalent to $\widetilde A_\infty(1)$ 
via the change of variable $x\to x^a$. Since $q\to 1$ as $p\to 1$ (for fixed $s$), and since the $q$-Bessel function $J_\nu(x;q)$ degenerates in this limit to the classical Bessel function $J_\nu(x)$, the eigenvalues of $\widetilde A_{\infty}(a)$ are proportional to inverse squared zeros of $J_\nu$, where $\nu=-\frac{ds}{ds+1}$. 

Furthermore, $\widetilde A_\infty(a)$ is inverse to a second order differential operator.  
Namely, consider the differential operator 
$$
\mathcal L_{\nu}:=-\partial^2+\frac{\nu^2-\frac{1}{4}}{x^2}.
$$ 
Assume that $\nu>-1$ and consider the action of $\mathcal{L_\nu}$ on the space of functions $g$ on $[0,1]$ such that $gx^{-\nu-\frac{1}{2}}$ is a polynomial and $g(1)=0$.  This defines an essentially self-adjoint operator. Its eigenvalues and eigenvectors are found in the standard way (for $\nu=0$ this is the classical problem of vibrations of a circular membrane, i.e., the Dirichlet problem for the Laplacian on the disk). 
Namely, consider the solution of the equation $\mathcal L_\nu\psi=\psi$ which behaves as $x^{\nu+\frac{1}{2}}$ at $0$. This solution is $x^{\frac{1}{2}}J_\nu(x)$. Thus the solution of $\mathcal L_\nu\psi=\lambda^2\psi$ with such behavior at $0$, up to scaling, is $x^{\frac{1}{2}}J_\nu(\lambda x)$. So the eigenvalues of $\mathcal L_\nu$ are squared zeros of the Bessel function $J_{\nu}(x)$. 

On the other hand, let us compute the Green function of $\mathcal L_\nu$. This is also standard. We need to find the solution of the equation $\mathcal L_\nu^xg(x,y)=\delta(x-y)$ with the above boundary conditions. So $g(x,y)=b(y)(1-y^{-2\nu})x^{\nu+\frac{1}{2}}$ when $x<y$ 
and $g(x,y)=b(y)(x^{\nu+\frac{1}{2}}-x^{-\nu+\frac{1}{2}})$ 
when $x>y$ with gluing condition
$b(y)=-\frac{y^{\nu+\frac{1}{2}}}{2\nu}.$
Thus 
$$
g(x,y)=\frac{1}{2\nu}(\max(x,y)^{-2\nu}-1)(xy)^{\nu+\frac{1}{2}}.
$$
So we have 
$$
\widetilde A_\infty(2\nu)^{-1}=\mathcal L_\nu 
$$ 
where $\nu=-\frac{ds}{ds+1}$. This gives another way to see that eigenvalues of $\mathcal L_\nu$ are squared zeros of the Bessel function $J_\nu(x)$. 

\begin{remark} The differential operator $\mathcal L_\nu$ is the continuous limit of the difference operator $L_\nu$ considered in the previous subsection. 
\end{remark} 

\subsection{The case when $Z_0$ is a constant}
Assume now that $Q\in \Bbb Z_p[x]$ and the reduction $\overline Q$ of $Q$ modulo $p$ also has degree $d$. In this case the function $Z_0(Q,\chi,s,z)$ defined by \eqref{zeq},\eqref{z0eq} is constant as a function of $z$, 
i.e., 
$$
Z_0(Q,\chi,s,z)=\zeta(Q,\chi,s)-\delta_{1,\chi}.
$$ 
The case $\chi\ne 1$ is trivial, so we now consider the case $\chi=1$. 

For brevity we will denote $\zeta(Q,1,s)-1$ by $\beta=\beta(s)$. The case when $\beta(s)=0$ has already been discussed in Subsection \ref{ex1}, so we assume that $\beta(s)\ne 0$. 
For instance, if $\overline Q$ has $m$ zeros in $\Bbb F_p$ all of which are simple then 
by the computation in Example \ref{deg1} we have 
$$
\beta(s)=-m\frac{p^{-1}(1-p^{-s})}{(1-p^{-1})(1-p^{-s-1})}. 
$$
We have 
$$
R(z)=\frac{1}{1-z}-\frac{1}{1-pqz}+\beta=\beta
\frac{(1-az)(1-a^{-1}pqz)}{(1-z)(1-pqz)}.
$$
for $a$ found from the quadratic equation 
$$
\beta(1-a)(1-a^{-1}pq)=1-pq.
$$
Thus eigenvectors of $A_{P,s,1}$ are solutions of the $q$-difference equation 
$$
\lambda F(z)=\frac{\beta(1-az)(1-a^{-1}pqz)}{(1-z)(1-pqz)}F(qz)+\frac{F(p^{-1})}{1-pqz},
$$
Setting $z=0$, we have 
$$
(\lambda-\beta)F(0)=F(p^{-1}).
$$ 
First consider the case when $F(p^{-1})\ne 0$. Then $F(0)\ne 0$, so setting $F(0)=1$, we get 
$\lambda-\beta=F(p^{-1})$. Thus, using that $\lambda\ne \beta$, we have 
$$
F(z)=\frac{1-\beta\lambda^{-1} }{1-pqz}+\beta\lambda^{-1} \frac{(1-az)(1-a^{-1}pqz)}{(1-z)(1-pqz)}F(qz).
$$
Hence, iterating, we obtain 
$$
F(z)=\frac{1-\beta\lambda^{-1} }{1-pqz}+\beta\lambda^{-1} \frac{(1-az)(1-a^{-1}pqz)}{(1-z)(1-pqz)}\frac{1-\beta\lambda^{-1} }{1-pq^2z}
$$
$$
+\beta^2\lambda^{-2}\frac{(1-az)(1-aqz)(1-a^{-1}pqz)(1-a^{-1}pq^2z)}{(1-z)(1-qz)(1-pqz)(1-pq^2z)}\frac{1-\beta\lambda^{-1} }{1-pq^3z}+...
$$
Substituting $z=p^{-1}$, we have 
$$
1=\lambda^{-1}\frac{1}{1-q}+\beta\lambda^{-2}\frac{(1-ap^{-1})(1-a^{-1}q)}{(1-p^{-1})(1-q)}\frac{1}{1-q^2}
$$
$$
+\beta^2\lambda^{-3}\frac{(1-ap^{-1})(1-ap^{-1}q)(1-a^{-1}q)(1-a^{-1}q^2)}{(1-p^{-1})(1-p^{-1}q)(1-q)(1-q^2)}\frac{1}{1-q^3}+...
$$
Since $\beta=\frac{1-pq}{(1-a)(1-a^{-1}pq)}=-\frac{1-p^{-1}q^{-1}}{(1-a^{-1})(1-ap^{-1}q^{-1})}$, this can be written as
\begin{equation}\label{eqeig} 
{}_2\phi_1(ap^{-1}q^{-1}, a^{-1}; p^{-1}q^{-1}; q,\beta\lambda^{-1}  a^{-1})=0.
\end{equation} 
So we see that the eigenvalues of $A$ are the solutions 
of the equation \eqref{eqeig}. 

It remains to consider the case $F(p^{-1})=0$. In this case 
we get 
$$
\lambda F(z)=\frac{\beta(1-az)(1-a^{-1}pqz)}{(1-z)(1-pqz)}F(qz).
$$
Setting $z=p^{-1}q^{k-1}$, $k\ge 1$, we see that $F(p^{-1}q^{k-1})=0$ implies $F(p^{-1}q^{k})=0$ unless $ap^{-1}q^{k-1}=1$ or $a^{-1}q^{k}=1$. If $F(p^{-1}q^{k})=0$ for all $k\ge 0$ then $F=0$, so we see that we must have $a=pq^{-k+1}$ or $a=q^{k}$ for some $k\ge 1$. In this case 
$$
\beta=\frac{1-pq}{(1-q^{k})(1-pq^{-k+1})},
$$ 
and we have eigenvalues $\beta q^j$, $j\ge 0$, whose eigenvectors have the form $z^j+O(z^{j+1})$, in addition to $k$ eigenvalues which solve equation \eqref{eqeig} 
(as the $q$-hypergeometric function on the left hand side of this equation specializes to
a polynomial of degree $k$). 

So in this case the spectrum consists of two different parts, finite and infinite. 
The infinite part is fairly trivial ($\beta q^j$, $j\ge 0$), while 
the finite part is more interesting and consists of the numbers $\beta \gamma_i^{-1}$, where $\gamma_1,...,\gamma_k$ are the algebraic functions of $q$ 
which are zeros of the $k$-th little $q$-Jacobi polynomial
$p_k(q^{-1}x;p^{-1}q^{-2},1;q)$ ([KoS]). If $s\in \Bbb Q$, these are algebraic numbers. 

We note that as $k\to \infty$ (i.e., $\beta\to 0$), we find ourselves in the situation of Proposition \ref{eig1} for $r=d$. This recovers \cite{KSw}, Proposition A1, stating that 
the little $q$-Jacobi polynomials can be degenerated to the $q$-Bessel function $J(a; q,u)$.  

These two cases ($F(p^{-1})\ne 0$ and $F(p^{-1})=0$) can be nicely unified as follows. 
Recall that the function ${}_2\phi_1$ for generic parameters has at most simple poles at $1,q^{-1},q^{-2},...$ and no other singularities. Thus the function 
$$
{}_2\widetilde\phi_1(a,b,c,q,u):={}_2\phi_1(a,b,c,q,u)(u; q)_\infty
$$
is entire (it is obtained from ${}_2\phi_1$ by ``clearing denominators"). 

\begin{proposition} The characteristic function of $A_{P,s,1}$ is  
$$
h_{P,s,1}(u)={}_2\widetilde\phi_1(ap^{-1}q^{-1}, a^{-1}; p^{-1}q^{-1}; q,\beta u).
$$
Thus the eigenvalues of $A_{P,s,1}$ are the solutions of the equation 
$$
{}_2\widetilde\phi_1(ap^{-1}q^{-1}, a^{-1}; p^{-1}q^{-1}; q,\beta\lambda^{-1} )=0.
$$
\end{proposition} 

\begin{example} Let $P(x,y)=x-y$. Then $d=m=1$ and 
$$
\beta=\frac{1-pq}{(1-p)(1-q)}.
$$ 
So the equation for $a$ takes the form 
$$
(1-a)(1-a^{-1}pq)=(1-p)(1-q),
$$
which gives $a=p$ or $a=q$. We have  
$$
{}_2\phi_1(q^{-1}, p^{-1}; p^{-1}q^{-1}; q,u)=1-\frac{1-p}{1-pq}u. 
$$
(the first little $q$-Jacobi polynomial $p_1$). 
Thus for the eigenvalues we get the equation 
$$
 \left(1-\frac{1-p}{1-pq}\beta\lambda^{-1} \right)(\beta\lambda^{-1} ;q)_\infty.
$$
So the eigenvalues are 
\begin{equation}\label{eigenlist}
\lambda_0=\frac{1}{1-q},\ \lambda_{k+1}=\frac{1-pq}{(1-p)(1-q)}q^k,\ k\ge 0.
\end{equation} 
In fact, this is easy to see in a different way: the operator $A$ is the convolution with the function $\frac{|x|^s}{1-p^{-1}}$ on $\Bbb Z_p$, so its eigenvalues are the values of the Fourier transform of this function, which are easily shown to be exactly 
the list \eqref{eigenlist}. 
\end{example}

\subsection{The case $\chi\ne 1$ with $Z_0$ having a first order pole at $0$ and $\infty$}\label{Mori} Consider now the case $\chi\ne 1$ with $Z_0$ having the form 
$z^{-1}(1-az)(1-bz)$. This happens, for instance, 
in the case 
$$
P(x,y)=x^d-p^{-1}xyP_*(x,y)+y^d,
$$
where $P_*(x,y)\in \Bbb Z_p[x,y]$ is a homogeneous polynomial of degree 
$d-2$ whose coefficients of $x^{d-2}$ and $y^{d-2}$ have norm $1$. 
In this case 
$$
Q(x)=1-p^{-1}xQ_*(x)+x^d,
$$
where $Q_*(x)=P_*(x,1)$. 
Then we have 
$$
Z_0(P,\chi,s,z)=\zeta(1-x,\chi,s)(p^{-1}q^{-1}z+z^{-1})+p^s\zeta(Q_*,\chi,s).
$$
Let 
$$
\gamma=\gamma(s):=\frac{p^s\zeta(Q_*,\chi,s)}{\zeta(1-x,\chi,s)}
$$
and let 
$$
z^{-1}+\gamma +p^{-1}q^{-1}z=z^{-1}(1-az)(1-bz).
$$
Then after rescaling $z$ the eigenvalue equation for $A_{P,s,\chi}$ 
takes the form 
\begin{equation}\label{firstor}
\lambda F(z)=(z^{-1}(1-az)(1-bz)F(qz))_+.
\end{equation} 
So let us solve \eqref{firstor} for general $a,b$. 

We can rewrite \eqref{firstor} in the form
\begin{equation}\label{firstor1}
\lambda F(z)=z^{-1}(1-az)(1-bz)F(qz)-z^{-1}F(0). 
\end{equation} 
To solve \eqref{firstor1}, consider the deformed equation 
\begin{equation}\label{firstor2}
\lambda F(z)=-\frac{t^{-1}(1-az)(1-bz)F(qz)}{1-t^{-1}z}+\frac{t^{-1}(1-at)(1-bt)F(qt)}{1-t^{-1}z} 
\end{equation} 
which degenerates to \eqref{firstor} as $t\to 0$, and let us solve \eqref{firstor2}. 
We have 
$$
(1-at)(1-bt)F(qt)=(1+t\lambda)F(0).
$$
Assume first that $F(qt)\ne 0$. Then $F(0)\ne 0$, and setting $F(0)=1$, we get 
$$
F(z)=-\lambda^{-1}t^{-1}\frac{(1-az)(1-bz)}{1-t^{-1}z}F(qz)+\frac{1+\lambda^{-1}t^{-1}}{1-t^{-1}z}. 
$$
Iterating this, we obtain 
$$
F(z)=\frac{1+\lambda^{-1}t^{-1}}{1-t^{-1}z}-\lambda^{-1}t^{-1}\frac{(1-az)(1-bz)}{1-t^{-1}z}\frac{1+\lambda^{-1}t^{-1}}{1-t^{-1}qz}
$$
$$
+\lambda^{-2}t^{-2}\frac{(1-az)(1-aqz)(1-bz)(1-bqz)}{(1-t^{-1}z)(1-t^{-1}qz)}\frac{1+\lambda^{-1}t^{-1}}{1-t^{-1}q^2z}-...
$$
Thus, setting $z=qt$, we have 
$$
\frac{1+t\lambda}{(1-at)(1-bt)}=F(qt)=\frac{1+\lambda^{-1}t^{-1}}{1-q}-\lambda^{-1}t^{-1}\frac{(1-aqt)(1-bqt)}{1-q}\frac{1+\lambda^{-1}t^{-1}}{1-q^2}
$$
$$
+\lambda^{-2}t^{-2}\frac{(1-aqt)(1-aq^2t)(1-bqt)(1-bq^2t)}{(1-q)(1-q^2)}\frac{1+\lambda^{-1}t^{-1}}{1-q^3}-...
$$
So we get that the eigenvalues $\lambda$ are the solutions of the equation
$$
{}_2\phi_1(at,bt;\ 0;\ -\lambda^{-1}t^{-1})=0.
$$
Now we have to determine what happens to this equation as $t\to 0$. This is done in \cite{Mo}, but we give the details for the reader's convenience. 
Recall the Watson connection formula for the Heine $q$-hypergeometric function 
${}_2\phi_1$ (\cite{GR}): 
$$
{}_2\phi_1(a,b;c;q,x) =
$$
$$
\frac{(b,\frac{c}{a};q)_\infty}{(c,\frac{b}{a};q)_\infty}\frac{\theta(-ax;q)}{\theta(-x,q)}{}_2\phi_1(a,\tfrac{aq}{c};\tfrac{aq}{b};q,\tfrac{cq}{abx})+\frac{(a,\frac{c}{b};q)_\infty}{(c,\frac{a}{b};q)_\infty}\frac{\theta(-bx;q)}{\theta(-x;q)}{}_2\phi_1(b,\tfrac{bq}{c};\tfrac{bq}{a};q,\tfrac{cq}{abx}).
$$
So replacing $a$ with $at$, $b$ with $bt$, 
and $x$ with $-ut^{-1}$, we get 
$$
{}_2\phi_1(at,bt;c;q,-ut^{-1}) =
\frac{(bt,\frac{c}{at};q)_\infty}{(c,\frac{b}{a};q)_\infty}\frac{\theta(au;q)}{\theta(ut^{-1},q)}{}_2\phi_1(at,\tfrac{atq}{c};\tfrac{aq}{b};q,-\tfrac{cq}{abtu})
$$
$$
+\frac{(at,\frac{c}{bt};q)_\infty}{(c,\frac{a}{b};q)_\infty}\frac{\theta(bu;q)}{\theta(ut^{-1};q)}{}_2\phi_1(bt,\tfrac{btq}{c};\tfrac{bq}{a};q,-\tfrac{cq}{abtu}).
$$
Now let us send $c$ to $0$ and use that 
$$
\lim_{c\to 0}{}_2\phi_1(\alpha,\tfrac{\beta}{c};\gamma; q,cz)={}_1\phi_1(\alpha;\gamma;q,\beta z)
$$
(see \cite{GR}). This yields
$$
{}_2\phi_1(at,bt;0;q,-ut^{-1}) =
$$
$$
\frac{(bt;q)_\infty}{(\frac{b}{a};q)_\infty}\frac{\theta(au;q)}{\theta(ut^{-1},q)}{}_1\phi_1(at;\tfrac{aq}{b};q,-\tfrac{q^2}{bu})+\frac{(at;q)_\infty}{(\frac{a}{b};q)_\infty}\frac{\theta(bu;q)}{\theta(ut^{-1};q)}{}_1\phi_1(bt;\tfrac{bq}{a};q,-\tfrac{q^2}{au}).
$$
Thus the equation ${}_2\phi_1(at,bt;0;q,-ut^{-1})=0$ yields
$$
\frac{(bt;q)_\infty}{(\frac{b}{a};q)_\infty}\theta(au;q){}_1\phi_1(at;\tfrac{aq}{b};q,-\tfrac{q^2}{bu})+\frac{(at;q)_\infty}{(\frac{a}{b};q)_\infty}\theta(bu;q){}_1\phi_1(bt;\tfrac{bq}{a};q,-\tfrac{q^2}{au})=0
$$
Now sending $t$ to $0$, we get 

\begin{proposition}\label{eigefor} The 
eigenvalues of $A_{P,s,\chi}$ are the solutions of the equation 
$$
E(a,b;q,\lambda^{-1})=0,
$$
where 
\begin{equation}\label{firstor3} 
E(a,b;q,u):=\frac{\theta(au;q)J(\tfrac{aq}{b};q,-\tfrac{q^2}{bu})}{(\tfrac{b}{a};q)_\infty(q;q)_\infty}
+\frac{\theta(bu;q)J(\tfrac{bq}{a};q,-\tfrac{q^2}{au})}{(\tfrac{a}{b};q)_\infty(q,q)_\infty}. 
\end{equation} 
\end{proposition} 

Indeed, the case $F(qt)=0$ is irrelevant for us since then we must have 
$\lambda=t^{-1}q^m$ for some $m$, and these eigenvalues, if present, go to $\infty$
as $t\to 0$, since $m$ has to remain bounded.  

As before, by growth considerations Proposition \ref{eigefor} implies that 
$$
\det(1-uA_{P,\chi,s})=Cu^mE(a,b;q;u)
$$
for some $C\ne 0$, $m\in \Bbb Z$. In particular, this means that 
$E(a,b;q;u)$ is meromorphic at $0$. Note also that 
the two summands in $E$ form a basis in the space of solutions 
of the $q$-Bessel difference equation transformed by $u\mapsto u^{-1}$ over the field of elliptic functions:
\begin{equation}\label{diffeqn}
abqu^2K(q^2u)-(1+(a+b)u)K(qu)+K(u)=0. 
\end{equation} 
This equation is irregular at $0$, but it has a single, up to scaling, solution $\widetilde E(u)$ meromorphic at $0$, which can be found by the power series method. 
Moreover, it is easy to see that $\widetilde E(u)$ is in fact holomorphic 
at $0$ and does not vanish there (as equation \eqref{diffeqn} degenerates to the equation $K(qu)=K(u)$ near $u=0$). So up to scaling we must have 
$\widetilde E(u)=E(u)$, hence $m=0$. Thus we get 

\begin{corollary} The function $E(a,b; q,u)$ extends holomorphically 
to the origin with $E(a,b; q,0)\ne 0$, and 
$$
\det(1-uA_{P,s,\chi})=\frac{E(a,b; q,u)}{E(a,b;q,0)}.
$$
\end{corollary} 

In fact, it is shown in \cite{Mo}, Theorem 1 that $E(a,b;q,0)=1$. 
Thus one has 
$$
\det(1-uA_{P,s,\chi})=E(a,b; q,u).
$$

\section{The non-homogeneous case}

\subsection{The trace class property} 

Now let $P\in \Bbb Q_p[x,y]$ be any polynomial (not necessarily homogeneous). Let us first show that if ${\rm Re}(s)$ is sufficiently large then the operator $A_{P,s}$ is trace class. 
First recall the following well known lemma. 

\begin{lemma}\label{trc} Let $A$ be a bounded operator on a Hilbert space $H$ 
and let $A_k\to A$ be a sequence of bounded operators on $H$ of ranks $\le r_k$, $k\ge 0$, such that $r_{k-1}\le r_{k}$. If $M:=\sum_{k\ge 1} r_{k}\norm{A_{k-1}-A}<\infty$
then $A$ is trace class. 
\end{lemma} 

\begin{proof} Recall that for any bounded operator $B$ of rank $r$ 
we have ${\rm Tr}(B)\le r\norm{B}$. Now, we have 
$$
A=\lim_{k\to \infty}A_k=A_0+(A_1-A_0)+....+(A_{k}-A_{k-1})+...,  
$$
an absolutely convergent series with respect to the operator norm. So given a finite rank operator $C$, 
we have 
$$
{\rm Tr}(AC)={\rm Tr}(A_0C)+\sum_{k\ge 1}{\rm Tr}((A_{k}-A_{k-1})C). 
$$
Now, the rank of $A_{k}-A_{k-1}$ is at most $r_k+r_{k-1}\le 2r_k$, while its norm is at most $\norm{A_{k-1}-A}+\norm{A_k-A}$. Thus 
$$
|{\rm Tr}(AC)|\le (\norm{A_0}+\sum_{k\ge 1} 2r_k(\norm{A_{k-1}-A}+\norm{A_k-A}))\norm{C}.
$$
Hence
$$
{\rm Tr}|A|\le \norm{A_0}+\sum_{k\ge 1} 2r_k(\norm{A_{k-1}-A}+\norm{A_k-A})\le 
\norm{A_0}+4M.
$$
\end{proof} 

Now let $\gamma_P$ be the log-canonical threshold of $P$ (\cite{Mu}), i.e.,
$$
\gamma_P={\rm sup}(\gamma: {\rm vol}(|P(x,y)|<t)=O(t^\gamma),t\to 0).
$$

\begin{proposition}\label{tracl} If ${\rm Re}(s)>1-\frac{1}{2}\gamma_P$ then $A_{P,s}$ is trace class. 
\end{proposition} 

\begin{proof} 
Let ${\rm Re}(s)=\rho>0$. Let $A_k$ be the integral operator 
on $\Bbb Z_p$ with kernel $D_k(x,y):=\frac{1}{1-p^{-1}}\max(p^{-k},|P(x,y)|)^s$.
It is clear that $\norm{A-A_k}$ is dominated by its Hilbert-Schmidt norm, 
hence 
$$
\norm{A-A_k}\le C p^{-k(\rho+\frac{1}{2}\gamma_P)}. 
$$
for large $k$. 
On the other hand, we have 
$$
P(x,y)-P(x',y')=P_1(x,y)(x-x')+P_2(x',y)(y-y')
$$
for some polynomials $P_1,P_2$, so 
$$
|P(x,y)-P(x',y')|< p^{\ell}\max(|x-x'|,|y-y'|)
$$
for some $\ell$. This implies that if $|P(x,y)|=p^{-k}$ and 
$|x'-x|,|y'-y|\le p^{-k-\ell}$ then $|P(x',y')|=|P(x,y)|$.
Thus if $|x'-x|,|y'-y|\le p^{-k-\ell}$ then $D_k(x,y)=D_k(x',y')$. 
We conclude that the rank of $A_k$ is $\le p^{k+\ell}$.  
Thus by Lemma \ref{trc}, $A$ is trace class if $\rho>1-\frac{1}{2}\gamma_P$. 
\end{proof} 

For example, if $P$ has zeros in $\Bbb Z_p^2$ but they are all smooth points of 
the curve $P(x,y)=0$ then $\gamma_P=1$ and $A_{P,s}$ is trace class whenever ${\rm Re}(s)>\frac{1}{2}$. 

\subsection{An example} Proposition \ref{tracl} allows us to define the characteristic function $h_{P,s}(u)$ for a general polynomial $P$. Computing this function, however, seems difficult in general, as we no longer have a decomposition into the sectors $H_\chi$, and even when we do, the answer goes beyond the theory of $q$-hypergeometric functions. 

To demonstrate what may happen in the non-homogeneous case, consider the polynomial 
$$
P(x,y)=x^{2d}+y^{2}, 
$$ 
where $d>1$ and $p=4k+3$. We have $|P(x,y)|=\max(|x|^d,|y|)$. 
Thus the corresponding operator $A_{P,s}$ still commutes with 
$\Bbb Z_p^\times$ and $A_{P,s,\chi}=0$ if $\chi\ne 1$, while  
the matrix of $A=A_{P,s,1}$ is 
$$
a_{mn}=p^{-\frac{m+n}{2}-2s\min(m,dn)}. 
$$

For simplicity assume that $s\ge 0$. In terms of analytic functions $F(z)$ we get 
$$
(AF)(z)=\sum_{n\ge 0}\sum_{m>dn}z^mp^{-2dns-\frac{n}{2}}f_n
+\sum_{n\ge 0}\sum_{m\le dn}z^mp^{-2ms-\frac{n}{2}}f_n=
$$
$$
\frac{(1-p^{-2s})z}{(1-z)(1-p^{-2s}z)}F(p^{-2ds-1}z^d)+\frac{F(p^{-1})}{1-p^{-2s}z}.
$$
So the eigenvalue equation looks like 
$$
\lambda F(z)=\frac{(1-p^{-2s})z}{(1-z)(1-p^{-2s}z)}F(p^{-2ds-1}z^d)+\frac{F(p^{-1})}{1-p^{-2s}z}.
$$
Thus $\lambda F(0)=F(p^{-1})$. If $F(p^{-1})=0$ then $F=0$, so $F(p^{-1})\ne 0$, hence $F(0)\ne 0$, $\lambda\ne 0$. So we may set $F(0)=1$, which gives $\lambda=F(p^{-1})$. Hence we get 
$$
F(z)=\frac{1}{1-p^{-2s}z}+\lambda^{-1}\frac{(1-p^{-2s})z}{(1-z)(1-p^{-2s}z)}F(p^{-2ds-1}z^d).
$$
Let $s>-\frac{1}{2d}$ and 
$q=p^{-\frac{2s+1}{d-1}}$. Then, replacing $z$ by $p^{-1}q^{-d}z$, we have 
$$
F(p^{-1}q^{-d}z)=\frac{1}{1-q^{-1}z}+\lambda^{-1}\frac{(p^{-1}q^{-d}-q^{-1})z}{(1-p^{-1}q^{-d}z)(1-q^{-1}z)}F(p^{-1}q^{-d}z^d).
$$
So setting $G(z):=F(p^{-1}q^{-d}z)$ and $u=\lambda^{-1}$, we get 
\begin{equation}\label{Geq}
G(z)=\frac{1}{1-q^{-1}z}+u\frac{(p^{-1}q^{-d}-q^{-1})z}{(1-p^{-1}q^{-d}z)(1-q^{-1}z)}G(z^d).
\end{equation} 
This shows that $G$ is a {\bf $d$-Mahler series} (of degree $2$) in the sense of \cite{BCCD}, which is a generalization of the notion of $d$-automatic and $d$-regular series (see \cite{AS}). 

Iterating \eqref{Geq}, for any $u$ we get 
a solution 
$$
G(z,u):=\frac{1}{1-q^{-1}z}+u\frac{(p^{-1}q^{-d}-q^{-1})z}{(1-p^{-1}q^{-d}z)(1-q^{-1}z)(1-q^{-1}z^d)}+
$$
$$
u^2
\frac
{(p^{-1}q^{-d}-q^{-1})^2z^{d+1}}
{(1-p^{-1}q^{-d}z)(1-p^{-1}q^{-d}z^d)(1-q^{-1}z)(1-q^{-1}z^d)(1-q^{-1}z^{d^2})}
+...;
$$
Namely, it is clear that this series converges (uniformly on every compact set of points $(z,u)$) in the region $|z|<\min(|q|,|pq^d|)$. Now set $z=q^d$, which belongs to this region. 
So setting $\beta:=u(q^{-1}-p^{-1}q^{-d})$, we obtain
$$
1=\frac{\beta q}{(1-p^{-1}q^{1-d})(1-q^{d-1})}-\frac{\beta^2 q^{d+1}}{(1-p^{-1}q^{1-d})(1-p^{-1})(1-q^{d-1})(1-q^{d^2-1})}
$$
$$
+\frac
{\beta^3q^{d^2+d+1}}
{(1-p^{-1}q^{1-d})(1-p^{-1})(1-p^{-1}q^{d^2-d})(1-q^{d-1})(1-q^{d^2-1})(1-q^{d^3-1})}
+...
$$
Define the ``automatic Pochhammer symbol" 
$$
[a;q]_{d,n}=(1-aq)(1-aq^d)(1-aq^{d^2})....(1-aq^{d^{n-1}}). 
$$
Let 
$$
K(a; q,u)=\sum_{n=0}^\infty \frac{(-1)^nq^{\frac{d^n-1}{d-1}}u^n}{[a;q]_{d,n}[q^{-1};q^d]_{d,n}}
$$
(an entire function). 
We obtain 

\begin{proposition} We have 
$$
\det (1-uA_{P,s,1})=K(p^{-1}q^{-d}; q, (q^{-1}-p^{-1}q^{-d})u).
$$
Thus the eigenvalues $A_{P,s,1}$ are the solutions of the equation 
$$
K(p^{-1}q^{-d}; q, (q^{-1}-p^{-1}q^{-d})\lambda^{-1})=0.
$$
\end{proposition} 

\begin{proof} Indeed, the right hand side is an entire function of order $0$ since its Taylor coefficients are rapidly decaying. 
\end{proof}

\end{document}